\documentclass{amsart}
\usepackage{graphicx}
\usepackage{amssymb}
\newtheorem{thm}{Theorem}[section]
\newtheorem{cor}[thm]{Corollary}
\newtheorem{lemma}[thm]{Lemma}

\newtheorem{prop}[thm]{Proposition}
\theoremstyle{definition}

\newtheorem{remark}[thm]{Remark}
\theoremstyle{question}

\theoremstyle{Conjecture}
\newtheorem{con}[thm]{Conjecture}
\theoremstyle{Problem}

\numberwithin{equation}{section}

\begin{document}

\title {The influence of the nilpotentizer on group structure}%
\author{N. Ahmadkhah and M. Zarrin}%

\address{Department of Mathematics, University of Kurdistan, P.O. Box: 416, Sanandaj, Iran}%
 \email{N.Ahmadkhah20@gmail.com}
 \address{Department of Mathematics, Texas State University, 601 University Drive, San Marcos, TX, 78666, USA}
 \email{m.zarrin@txstate.edu}
\begin{abstract}
For a finite group $G$ and an element $x\in G$, the subset $$ nil_G(x)=\{y\in G \mid <x,y>~~ is ~~ nilpotent\}$$
 is called nilpotentizer of $x$ in $G$. In this paper, we give two solvabilty criteria  for a finite group by the structure and the size
  of nilpotentizer of an element on finite group. In fact, we show that if there exists an element $x$ of $G$ such that $nil_G(x)$ 
  generates a maximal subgroup of $G$ and the simple commutator of weight $2 ~~or ~~3$ of elements of $nil_G(x)$ is equal to 
  $1$  or $|nil_G(x)|= p^n$, where $p$ is prime and $n=1, 2$. Then $G$ is a solvable group.  \\



\end{abstract}
\maketitle

\section{\textbf{ Introduction and Basic results}}
Abdollahi and Zarrin in \cite{AZ} introduced the nilpotentizer of an element $x$ of a group $G$ as follows: 

$$ nil_G(x)=\{y\in G \mid <x,y>~~ is ~~ nilpotent\},$$
and also $$nil(G)=\{x\in G \mid <x,y>~~ is ~~ nilpotent ~~for~~ all~~  y ~~in ~~G\}$$
that is called  nilpotentizer of $G$.
We can put it in this way, $nil_G(x)$ and $nil(G)$ are generalizations of centeralizer of an element, say $C_G(x)$ and center of a group $G$, say $Z(G)$. Means, $C_G(x)\subseteq nil_G(x)$ and $Z(G)\subseteq nil(G).$ For more information regarding $nil_G(x)$, see \cite{JMR}.

Note that for all elements $x\in G$, we have $nil_G(x)=G=nil(G)$ when $G$ is weakly nilpotent (i.e., every two generated subgroup of G is nilpotent, see also \cite{zw}). 
 But in general $nil_G(x)$ is not subgroup. For example, in the symmetric group $S_4$ of degree $4$, $nil_{S_4}((12)(34))$  is not subgroup;  and also  it is not known if $nil(G)$  
 is a subgroup, but in many  cases it is (see  Lemma 3.3 and  Proposition 2.1, \cite{AZ}).

  For that reason, the authors defined an $n$-group, where a group $G$ is an $n$-group if $nil_G(x)$ is a subgroup of $G$ for every $x\in G$.

The authors in \cite{AZ, Z} proved a few interesting results on the nilpotentizers. Here we collect these properties for reader convenience, as follows: 

\begin{thm}\label{n0}  \rm {(}see \cite{AZ})
Let $G$ be a group and $x\in G$. Then
	\item [(1)]$<x>\subseteq(<x>, Z(G))\subseteq C_G(x)\subseteq nil_G(x)$.
	\item [(2)]$nil_G(x)$ is the union of all maximal nilpotent subgroups of $G$ containing $x$.
	\item [(3)]$|nil_G(x)|$ is divisible by $|x|$.
	\item[(4)]If $N$ is a normal subgroup of $G$ and $x,y\in G$, then \label{n1} \\\\
	   A.  $\frac{nil_G(x)N}{N}\subseteq nil_\frac{G}{N}(xN)$. (Obviousely here $\frac{nil_G(x)N}{N}:= \{yN\mid y\in nil_G(x)N\})$.\\
	  B.  $nil_\frac{G}{K}(xK)=\frac{nil_G(x)}{K}$, where $K$ is a normal subgroup of $G$ with $K\le Z^*(G)$. \rm{(}$Z^*(G)$ is denoted the hypercenter of $G$\rm{)}.

\end{thm}

Those kind of interesting properties motivate us to continue the investigation of the influenec of nilpotentizer of an element on finite group.\\

Flavell \cite{P1} proved that a finite group is solvable if and only if every two elements generate a solvable group. In this result all the elements of group play an important role.
Then Flavell in \cite{P2}, proved that in a conjugacy class $\mathcal C$ of a finite group, if there exists a constant $k$  such that every $k$ elements of $\mathcal C$ generate a solvable subgroup,
then $\mathcal C$ generates a solvable subgroup. As in general, nilpotentizers of elements are not subgroups, in Section 2, we try to find some results, like  Flavell's results, regarding the influenece of the structure of 
nilpotentizers of elements on solvability of a finite group. Also it is an important problem to find algebraic conditions on the elements of a single nilpotentizer determining restrictions on the structure of the whole group. 

In Section $2$, first we prove that if for an element $x\in G$ there exists a constant $n$  such that every $n$ elements of $nil_G(x)$ generate a nilpotent subgroup of class at most $n-1$, then $nil_G(x)$ is a maximal nilpotent subgroup of class at most $n-1$. Then we obtain a  solvibility criterion for a group $G$ by the structure of its nilpotentizers. In fact, we show that: if there exists $x\in G$, such that $nil_G(x)$ generates a maximal subgroup of $G$ and $[y_1, y_2, y_3]=1$ for every $y_1, y_2, y_3\in nil_G(x)$, then $G$ is solvable.

\begin{thm}\label{n}
	Let $G$ be a group and there exists an element $x$ of $G$ such that $nil_G(x)$ generates a maximal subgroup of $G$.
	\begin{itemize}
		\item [(1)]Assume that the elements of $nil_G(x)$ commute pairwises, then $G$ is solvable.           
		\item [(2)]If $[y_1, y_2, y_3]=1$ for every $y_1, y_2, y_3\in nil_G(x)$, then $G$ is solvable.
	\end{itemize}
\end{thm}

Finally, we provide another solvibility criterion for a finite group by the size of nilpotentizer of an element of the group, as follows:

\begin{thm}\label{n7}
	Let $G$ be a non-solvable group and $x\in G$ in which $nil_G(x)$ generates a maximal subgroup of $G$. Then $|nil_G(x)|\neq p^n$, for any prime $p$ and $n=1,~~ 2$.
	
\end{thm}

All groups considered in the present paper are supposed to be finite and we use the usual notation, the cardinality of a set $X$ is denoted by $|X|$. Moreover, for an element $x\in G$, $o(x)$ is denoted 
the order of $x$, and $C_G(x)$ the centralizer of $x$ in $G$. For a subgroup $H$ of $G$, $C_G(H)$ is denoted the centralizer of $H$ in $G$, the Fitting subgroup of $G$ is denoted by $F(G)$.\\

\section{\textbf{Proofs}}

 First, we give two elementry properties of nilpotentizers and then we prove our main results. 
  \begin{lemma}
 	Let $H$ and $K$ be two subgroups of $G$ such that $G=HK$ and $[H,K]=1$. If $x\in H$, then $nil_G(x)=K nil_H(x)$.
 \end{lemma}
 
 \begin{proof}
 	Since $<x,kh>\le <x,h>$ for some $h\in H , k\in K$, so $K nil_H(x) \subseteq nil_G(x)$. Now assume that $y\in nil_G(x)$, then $<x,y>=<x, hk>$ is 
 	nilpotent where $y= hk$ for some $h\in H, k\in K$. On the other hand by hypothesis $<x, hk>\le C_G(k)$ and $<x, h>\le<x, hk, k>$ is nilpotent.
 	 It follows that $h\in nil_H(x)$ and $y\in K nil_H(x)$. The result follows.
 \end{proof}
 
 \begin{cor}
        Let $H$ and $K$ be two subgroups of $G$ such that $G=HK$ and $[H,K]=1$. If $x\in H$ and $H$ is a nilpotent, then $nil_G(x)=G$.    
 \end{cor}

  For a group $G$, and $a_1,..., a _n\in G$, we define inductively $[a_1,...,a_n]$  as
 follows: $[a_1] = a_1$ and
  $$[a_1,...,a_n]=[a_1,...,a_{n-1}] ^{-1}a_{n}^{-1}[a_1,...,a_{n-1}] a_n$$  for all $n > 1$

The following Lemma is one of the intriguing keys that one can find nilpotetizers that are subgroups. 
\begin{lemma}\label{n5}
Let $G$ be a group, $x\in G$ and $n\ge 2$. If $[ l_1,...,l_n ]=1$ for every $l_1,...,l_n\in nil_G(x)$, then $nil_G(x)$ is a subgroup. Moreover it is a maximal nilpotent of class at most $n-1$.
\end{lemma}

\begin{proof}
Assume that $x\in G$ and $y,z\in nil_G(x)$, we take $H=<x, y, z>$. So  by  (\cite{K}, 2.1.5), $\gamma _n(H)$ is generated by commutators of weight
 at least $n$ with set $\{x, y,z,x^{-1}, y^{-1}, z^{-1}\}$, hence by hypothesis $\gamma_n(H)=1$ and $H$ is nilpotent of class at most $n-1$. 
 Besides $<x, yz>\le H$, it implies that $yz\in nil_G(x)$ and it is a subgroup. Clearly $nil_G(x)$ is nilpotent of calss at most $n-1$ and  Case (2) of Theorem \ref{n0}  implies that $nil_G(x)$ is  a maximal nilpotent subgroup. 
\end{proof} 

In what follows, a famous theorem is needed which show the existence of a nilpotent maximal subgroup strongly affects the structure of a finite group (for more information see \cite{z9}).

\begin{thm}\cite{J}\label{n4}
	Let $G$ be a finite group having a nilpotent maximal subgroup $M$. If a $2$-Sylow subgroup of $M$ has class at most $2$, then $G$ is solvable.
\end{thm}

Theorem \ref{n4} states that, $2$-Sylow subgroups play a fundamental role in characterization of finite groups by their nilpotent maximal subgroups. 
In particular, when the order of the subgroup $M$ is odd, the Theorem \ref{n4} is immediately realized.\\

Now we are ready to prove Theorem 1.2. 

\text{{\bf Proof of Theorem 1.2.}}

\begin{itemize}

            \item [(1)]By appling Lemma \ref{n5}, $nil_G(x)$ is a subgroup, therefore $nil_G(x)=M$  is an abelian maximal subgroup and the result follows by Theorem \ref{n4}.    
           \item [(2)]By taking $k=3$ in Lemma \ref{n5}, $nil_G(x)$ is a subgroup, so $nil_G(x)=M$  is a nilpotent maximal subgroup of class at most $2$. It follows that $G$ is solvable by Theorem \ref{n4}. 
\end{itemize}

\begin{cor}\label{n15}
Let $G$ be a non-solvable group. If there exists an element $x$ of $G$ such that $nil_G(x)$ generates a maximal subgroup of $G$, then
\begin{itemize}
      \item [(1)]$<x>$ is properly contained in $nil_G(x)$.
      \item [(2)]There is a nilpotent subgroup $H$ of $G$  such that contains $<x>$ properly.
\end{itemize}
\end{cor}

\begin{remark}
According to Theorem \ref{n}, the following question naturally arises.
Let $G$ be a group, $x\in G $ and $[y_1,..., y_n]=1$ for every $y_1,..., y_n\in nil_G(x)$  with $n>3$. If $nil_G(x)$ is a maximal subgroup of $G$ (note that by Lemma \ref{n5}, $nil_G(x)$ is a subgroup), then is $G$ solvable?
Here we show that the answer to this question is negetive.\\
Let $G=PSL(2,17)$ and $S$ be a $2$-Sylow subgroup of $G$ and $x$ be an element of order $8$ in $S$. Then $nil_G(x)=S$ and $nil_G(x)$ is a nilpotent maximal subgroup of nilpotency class $3$.               
            
\end{remark}  
The above Remark, demonstrates  that the nilpotency class of the maximal subgroup in Janko's theorem is more important than it's  structure. 

As Theorem \ref{n} is not necessarily true for $n>3$. Here, we try to find the structure of a minimal counterexample, as Remark \ref{n14}, below.\\ 
We need the following theorems.

\begin{thm}\label{n9}\cite{Ro}
                 Suppose that $G$ is a finite non-solvable group having a nilpotent maximal subgroup $M$. If $Z(G)=1$, then $M$ is a $2$-Sylow subgroup of $G$.
\end{thm}

\begin{thm}\label{n10}\cite{B}
               Let $G$ be a finite non-solvable group having a nilpotent maximal subgroup. Let $L=F(G)$  be the Fitting subgroup of $G$. Then $\frac{G}{L}$ has
                a unique minimal normal subgroup $\frac{K}{L}$, which is a direct product of copies of a simple group with dihedral $2$-Sylow subgroups, and $\frac{G}{K}$ is a $2$-group.
\end{thm}

Furthermore, dihedral $2$-Sylow subgroups have been characterized by Gorenstein and Walter in the following theorem in \cite{G2}.

\begin{thm}\label{n11}\cite{G2}
           If $G$ is a simple group with dihedral $2$-Sylow subgroups, then either $G$ is isomorphic to the projective special linear group $PSL(2, q)$, $q$ odd and $q\ge 5$, or $G$  is isomorphic to the alternating group $A_7$.
\end{thm}

\begin{remark}\label{n14}
             Let $n\ge4$, and $G$ be a minimal (with respect to the order) non-solvable group having an element $x$ such that $nil_G(x)$ generates a maximal subgroup $M$ of $G$, $Z^*(G)\le M$ and $[ l_1,..., l_n]=1$ for every $l_1,..., l_n\in nil_G(x)$. Then $\frac{G}{F(G)}$ has a unique minimal normal subgroup $\frac{K}{F(G)}=S\times...\times S$ where $S$ is isomorphic either to $PSL(2, q)$, $q$ odd and $q\ge5$ or to the alternating group $A_7$ and $\frac{G}{K}$ is a $2$-group. In particular $nil_G(x)$ is a $2$-Sylow subgroup of $G$.
\end{remark}  
 
\begin{proof}
         Put $H=nil_G(x)$. By applying Lemma \ref{n5}, $H$ is a subgroup and it is nilpotent maximal of class at most $n-1$. Assume that $Z^*(G)\neq1$, 
         thus $|\frac{G}{Z^*(G)}|<|G|$ and $\frac{G}{Z^*(G)}$ is a group with $nil_{\frac{G}{Z^*(G)}}(xZ^*(G))$ which holds true under our hypothesis. Therefore
          $\frac{G}{Z^*(G)}$ is solvable, a contradiction, since $G$ is non-solvable. Hence $Z^*(G)=Z(G)=1$ and by Theorem \ref{n9}, $H$ is a $2$-sylow subgroup 
          of $G$ of order at most $2^n$. So from Theorem \ref{n10}, we imply that there exists a unique minimal normal subgroup $\frac{K}{F(G)}$ of $\frac{G}{F(G)}$ 
          such that $\frac{K}{F(G)}=S\times...\times S$, which $S$ is a non-abelian simple group with dihedral $2$-Sylow subgroups and $\frac{G}{K}$ is a $2$-group. 
          Thus Theorem \ref{n11} follows that, either $S$ is isomorphic to the projective special linear group $PSL(2, q)$, $q$ odd and $q\ge 5$, or it is isomorphic to the alternating group $A_7$.
 \end{proof}
      
Note that, when $F(G)=1$, the structure of group that are not solvable for $n\ge4$ is as follows:

\begin{remark}
            Let $n\ge4$, and $G$ be a minimal (with respect to the order) non-solvable group having an element $x$ such that $nil_G(x)$ generates a maximal subgroup of $G$ and $[ l_1,..., l_n]=1$ for every $l_1,..., l_n\in nil_G(x)$. Then $G$ has a unique minimal normal subgroup $K=S\times...\times S$  where $S$ is isomorphic either to $PSL(2, q)$, $q$ odd and $q\ge5$ or to the alternating group $A_7$. Moreover, $G=K<x>$ and $nil_G(x)$ is a $2$-Sylow subgroup of $G$.
\end{remark}

\begin{proof}
       Since $F(G)=Z^*(G)=Z(G)=1$, so by similar to the proof of Remark \ref{n14},  $G$ has a unique minimal normal subgroup $K=S\times...\times S$  where $S$ is isomorphic either to $PSL(2, q)$, $q$ odd and $q\ge5$ or to the alternating group $A_7$.\\
           Finally from Theorem 2.13 of \cite{I}, we imply that $x$ is not an involution element. Now let $H=<x,K>$. If $H$ is a proper subgroup of $G$, then since $nil_H(x)=H\cap nil_G(x)$  satisfies hypotheses, so $H$  is solvable. It follows that $G=<x, K>=K<x>$. Suppose that $P=M\cap K$  is a $2$-Sylow subgroup of $K$. By the Dedekind’s Modular Law we have
$$M=G\cap M=<x>K\cap M=<x>(K\cap M)= <x>P.$$ Therefore by Theorem \ref{n10}, $P$ is direct product of dihedral groups and the proof is complete.
\end{proof}

Now  we study the influence of $|nil_G(x)|$  on the structure of group $G$.\\ 

\begin{lemma}\label{n13}
	Let $G$ be a non-solvable group and $x\in G$ such that $nil_G(x)$ is a maximal subgroup of $G$. Then $nil_G(x)\neq p^n$ for all odd primes $p$ and every positive integer $n$.
\end{lemma}  

\begin{proof}
	Suppose, on the contrary, that $|nil_G(x)|=p^n$ for some odd prime $p$ and a positive integer $n$. Therefore $nil_G(x)$ is a $p$-Sylow subgroup of $G$ and $nil_G(x)$ is a nilpotent maximal subgroup of $G$ of odd prime power order. From Theorem \ref{n4}, $G$ is solvable, which is a contradiction.
\end{proof}

The following result shows that the nilpotentizer of an element of prime order has to be large enough when it is not equal to the centralizer.

\begin{lemma}\label{n6}
	Let $G$ be a group and $x\in G$ be of prime order. If $|nil_G(x)|\le p^2$, then $nil_G(x)=C_G(x)$. 
\end{lemma}  

\begin{proof}
	It is enough to prove that $nil_G(x)\subseteq C_G(x)$. Since $nil_G(x)$ is the union of all maximal nilpotent subgroup of $G$ containing $x$, so we show that if $M$ is a maximal nilpotent subgroup of $G$ containing $x$, then $M\subseteq C_G(x)$. Assume that $M$ is a maximal nilpotent subgroup containing $x$, hence $|M|\le p^2$. If $M$ is a $p$-group, then $M$ is abelian and $M\le C_G(x)$. Now suppose that $M$ is not $p$-group, thus $<x>$ is a $p$-Sylow subgroup of $M$ and $|M:<x>|<p$. Therefore $|M:C_M(<x>)|=1$, this implies that $M=C_M(x)\le C_G(x)$ and the result follows.  
\end{proof}

\text{{\bf Proof of Theorem 1.3.}}

	Assume, on the contrary, that $|nil_G(x)|=p$ for some prime $p$, so from Case (3) of Theorem \ref{n0}, we imply that $nil_G(x)=<x>$ and by Case (1) of Corollary \ref{n15}, a contradiction. Again, arguing by contradiction, suppose that, $|nil_G(x)|=p^2$ for some prime $p$. As $o(x)$ divides $p^2$, then Case (1) of Corollary \ref{n15} yields that $o(x)=p$, this follows that, $nil_G(x)=C_G(x)$ by Lemma \ref{n6} which is our final contradiction by Theorem \ref{n4}.

\begin{cor}\label{n12}
	Let $G$ be a non-solvable group and $x\in G$ where $nil_G(x)$ generates a maximal subgroup $M$ of $G$. If $Z^*(G)\neq 1$, $Z^*(G)\le M$, then
	\begin{itemize}
		\item [(1)]$|nil_G(x)|\neq p^3$ for every prime $p$.
		\item [(2)]$|nil_G(x)|\neq pq$ where $p, q$ are distinct primes.
	\end{itemize}
	
\end{cor}

\begin{proof}
	(1)~~Suppose, on the contrary, that there are $x\in G$ and peime $p$ with $|nil_G(x)|=p^3$, then Case (4, B) of Theorem \ref{n0} implies that $nil_\frac{G}{Z^*(G)}(xZ^*(G))$ is of order $p$ or $p^2$. Hence by Theorem \ref{n7}, $\frac{G}{Z^*(G)}$ is solvable. This gives the contradiction that $G$ is solvable.\\
	
	(2)~~Since $Z^*(G)\neq 1$ and from Case (4, B) of Theorem \ref{n0}, $nil_\frac{G}{Z^*(G)}(xZ^*(G))= \frac{nil_G(x)}{Z^*(G)}$, hence $nil_\frac{G}{Z^*(G)}(xZ^*(G))$ is of size $p$ or $q$, which is contradiction by Theorem \ref{n7}. 
\end{proof}

We guess that the condition $Z^*(G)=1$ is not necessary in Case (1) of Corollary \ref{n12}, when $p=2$.
\begin{con}
	If $|nil_G(x)|=8$ and $nil_G(x)$ generates a maximal subgroup of $G$, then $G$ is a solvable group. 
\end{con} 

\begin{prop}\label{n8}
	If $G$ is non-solvable group with self-centralizer elements of order $3$ and $nil_G(x)$ generates a maximal subgroup of $G$, then $|nil_G(x)|\neq6$ for every $x\in G$.
\end{prop}
\begin{proof}
	Assume that $|nil_G(x)|=6$. By Case (1) of Corollary \ref{n15}, $o(x)\neq 6$. Now as $o(x)$ divides $nil_G(x)$, the following two cases can be distinguished.\\

I. $o(x)=3$. Then Lemma \ref{n6} implies that $nil_G(x)=C_G(x)=<x>$, which is impossible by Case (1) of Corollary \ref{n15} .\\

II. $o(x)=2$. Assume that Q is a $2$-Sylow subgroup of $G$ containing $x$. $G$  is not $2$-nilpotent, because $G$ is not solvable. Therefore by (\cite{Rob},10.1.9), $Q$ is not cyclic, so $|Q|=4$ by Case (2) of Theorem \ref{n0}. Now let $y\in G$ be of order $2$, then $<x,y>$ is a dihedral group and $y\in nil_G(x)$. Thus all $2$-Sylow subgroups of $G$ is contained in $nil_G(x)$. If $n_2$ is the number of $2$-Sylow subgroups of $G$, since $G$ is not solvable, so $n_2>1$. It follows that $n_2\ge 3$ and hence $|nil_G(x)|>6$, which is a contradiction. 
	
\end{proof} 

Now we can set a lower bound for a nilpotentizer in a non-solvable group.

\begin{cor}
	Let $G$ be a non-solvable group and $x\in G$ where $nil_G(x)$ generates a maximal subgroup $M$ of $G$. If $Z^*(G)\neq 1$, $Z^*(G)\le M$, then $|nil_G(x)|\ge12$ 
\end{cor}
\begin{proof}
	By appling Theorem \ref{n7} and Corollary \ref{n12}, the proof is complete.
\end{proof}

\end{document}